\documentclass[12pt,oneside]{amsart}

\usepackage{microtype}

\usepackage{palatino, palatino}

\usepackage{amsmath,ifthen, amsfonts, amssymb,
srcltx, amsopn, color, enumerate} 
\usepackage[linktocpage=true]{hyperref}
\usepackage[cmtip,arrow]{xy}
\usepackage{pb-diagram, pb-xy}
\usepackage{overpic}

\newcommand{\showcomments}{yes}

\newsavebox{\commentbox}

\newcounter{ax}
\newtheorem{thm}{Theorem}[section]

\newtheorem{cor}[thm]{Corollary}
\newtheorem{conj}[thm]{Conjecture}

\newtheorem{prop}[thm]{Proposition}

\theoremstyle{definition}

\newtheorem{rem}[thm]{Remark}

\newtheorem{claim*}{Claim}

\newtheorem{question}[thm]{Question}

\makeatletter

\newcommand{\Rmnum}[1]{\mathbf{{\expandafter\@slowromancap\romannumeral #1@}}}

\makeatother

\let\oldmarginpar\marginpar
\renewcommand\marginpar[1]{\-\oldmarginpar[\raggedleft\footnotesize #1]{\raggedright\footnotesize #1}}

\newcommand{\tsh}[1]{\left\{\kern-.7ex\left\{#1\right\}\kern-.7ex\right\}}

\newcounter{enumitemp}

\newcommand{\PP}{\mathcal P}

 \usepackage{mathabx}

\begin{document}

\title[RACG Counterexample]{A counterexample to questions about boundaries, stability, and 
commensurability}
\author[J. Behrstock]{Jason Behrstock}
\address{Lehman College and The Graduate Center, CUNY, New York, New York, USA}
\email{jason.behrstock@lehman.cuny.edu}

\begin{abstract} We construct a family of right-angled Coxeter groups 
    which provide counter-examples to questions about the stable 
    boundary of a group, one-endedness of stable subgroups, and the 
    commensurability types of right-angled Coxeter groups.
\end{abstract}

\date{\today}

\maketitle

%\tableofcontents

\section*{Introduction}

In this short note we construct 
right-angled Coxeter groups with some interesting properties. 
These examples show that a number of questions in geometric group 
theory have more nuanced answers than originally expected. 
In particular, these examples resolve the following questions in the 
negative:

\begin{itemize}
    \item (Charney and Sisto): As is the case for right-angled Artin groups, do 
    all (non-relatively hyperbolic) 
    right-angled Coxeter group have totally disconnected 
    contracting boundary?

    \item (Taylor): Given that   
    all known quasigeodesically stable subgroups of the mapping 
    class group are virtually free, does it hold that in any  
    (non-relatively hyperbolic) group that all 
    quasigeodesically stable subgroups have more than one end?
    \item (Folk question): If a right-angled Coxeter group 
    has quadratic divergence, must it be virtually a right-angled 
    Artin group?
\end{itemize}

We describe a family of graphs, any one of which is the
presentation graph of a right-angled Coxeter group which provides a
counterexample to all three of the above questions.  We expect that 
in special cases, and perhaps in general with appropriate
modifications, there are interesting positive answers to these
questions; we hope this note will encourage the careful reader to 
formulate  
and prove 
such results.

The construction
we give was inspired by thinking about the simplicial boundary for the
Croke--Kleiner group, see
\cite[Example~5.12]{BehrstockHagen:cubulated1} and \cite{Hagen:tat}.  In the
process we give a quick introduction to a few topics of recent
interest in geometric group theory, for further details on these 
topics see also
\cite{AbbottBehrstockDurham, BehrstockDrutuMosher:thick,
Charney:RAAGsurvey, CharneySultan, Cordes:Morsesurvey, 
DurhamTaylor:stability, Tran:loxcox}.

\subsection*{Acknowledgments} Thanks to Ruth Charney, Alessandro 
Sisto, and Sam Taylor for sharing with me their interesting questions 
and to the referee for some helpful comments which led to Remark~\ref{rem:thanksref}. 
Also, thanks to Mark Hagen for many fun 
discussions relating to topics in this note.

\section{The construction}

Let $\Gamma_{n}$ be a graph with $2n$ vertices built in the following 
inductive way. Start with a pair of vertices $a_{1}, b_{1}$ with no 
edge between them. Given the graph $\Gamma_{n-1}$, obtain the graph 
$\Gamma_{n}$ by adding a new pair of vertices $a_{n},b_{n}$ to the 
graph $\Gamma_{n-1}$ and adding four new edges, one connecting each of 
$\{a_{n-1},b_{n-1}\}$ to each of  $\{a_{n},b_{n}\}$. See 
Figure~\ref{fig:eye}. 

\begin{figure}[h]
 \includegraphics[width=1.0\textwidth]{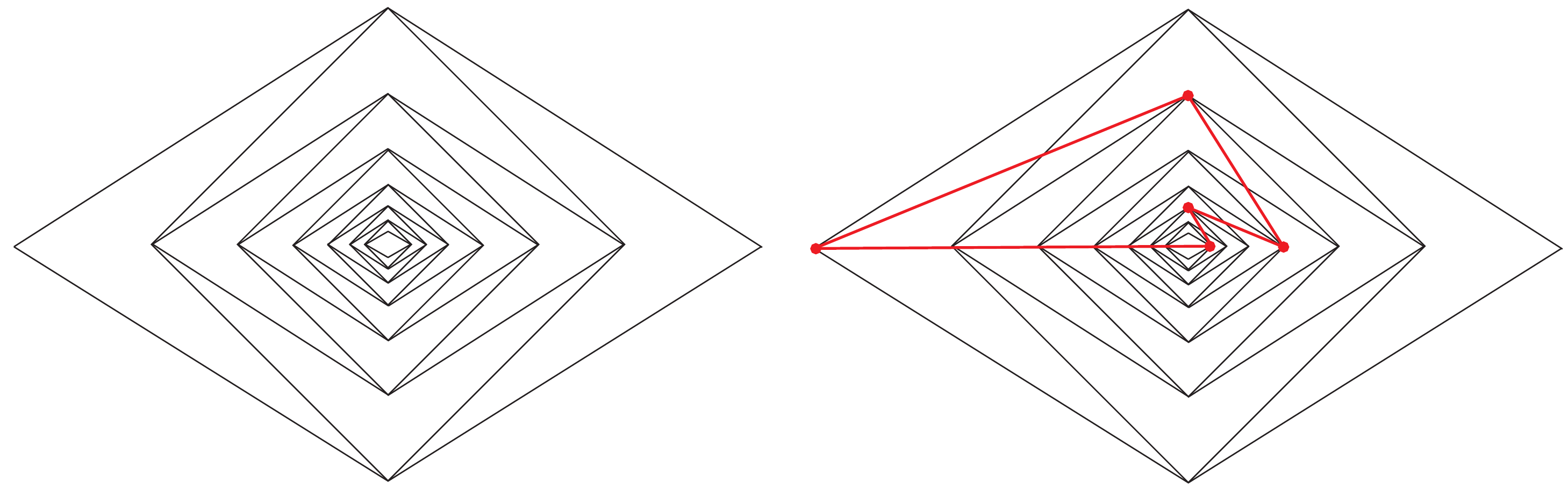}
 \caption{The graphs $\Gamma_{14}$ (left) and $\Gamma$ (right).}\label{fig:eye}
\end{figure}

Note that $\Gamma_{n}$ is a join if and only if $n\leq 3$. More 
generally, $a_{i},a_{j}$ are contained in a common join if and only  
if $|i-j|\leq 2$.

For any $m\geq 5$ choose $n$ sufficiently large so that there exists    
a set of 
points $\PP=\{p_{1},\ldots, p_{m}\}\subset \Gamma_{n}$ with the property that for 
each $1\leq i<j\leq m$ the points $p_{i}$ and $p_{j}$ 
are not contained in a common join in $\Gamma_{n}$. For example, 
in  $\Gamma_{14}$ we could choose the vertices 
$\PP=\{ a_{1},a_{4},a_{7},a_{10},a_{13}\}$. 
%$p_{1}=a_{0},p_{2}=a_{3},p_{3}=a_{6},p_{4}=a_{9},p_{5}=a_{12}$ 
For each $1\leq i< m$ 
add an edge between $p_{i}$ and $p_{i+1}$; also, add an edge between 
$p_{m}$ and $p_{1}$. Call this new graph $\Gamma$. There are many 
choices of $\Gamma$ depending on our choices of    
$n$, $m$, and $\PP$; for the following any choice will work.

Associated to any graph, one can construct the \emph{right-angled 
Coxeter group} with that presentation graph, this is the 
group whose defining presentation is given by: 
an order-two generator, for each vertex of the graph, and a 
commutation relation, between each of the generators associated 
to a pair of vertices connected by an edge.

Let $W$ denote the  
right-angled Coxeter group whose presentation graph is the $\Gamma$ 
constructed above.
In the next section, we record some key properties about the group 
$W$ and then, in the final section, apply this to the questions 
in the introduction.

\section{Properties}

\subsection{Quadratic divergence}

\begin{prop}\label{quaddiv}
    The group $W$ 
    has quadratic divergence. In particular, this group is not 
    relatively hyperbolic.
\end{prop}

\begin{proof}
    It is easily seen that the graph $\Gamma_{n}$ 
    has the property that each vertex is contained in at least one induced 
    square. It is also easy to verify that given any pair of 
    induced squares $S$, $S'$ in $\Gamma_{n}$, there exists a sequence of 
    induced squares $S=S_{1},S_{2},S_{3},\ldots, S_{k}=S'$ where for each 
    $1\leq i<k$ the squares $S_{i}$ and $S_{i+1}$ share 3 vertices in 
    common. This property, that there are enough squares to chain together 
    any pair of points, is called $\mathcal{CFS}$; it was 
    introduced in  
    \cite{DaniThomas:divcox} and studied further in 
    \cite{BehrstockHagenSusseFalgas:StructureRandom, Levcovitz:CFSdiv}.

    Since $\Gamma$ has the same vertex 
    set as $\Gamma_{n}$, and every induced square in $\Gamma_{n}$ is 
    still induced in $\Gamma$, it follows that $\Gamma$ also has the 
    $\mathcal{CFS}$ property.

    Given any graph with the $\mathcal{CFS}$ property and which is not a 
    join, the associated 
    right-angled Coxeter group has exactly quadratic 
    divergence, see \cite[Theorem~1.1]{DaniThomas:divcox} and 
    \cite[Proposition~3.1]{BehrstockHagenSusseFalgas:StructureRandom}. 
    
    The second statement in the proposition follows from the fact that any 
    relatively hyperbolic group has divergence which is at least 
    exponential 
    \cite[Theorem~1.3]{Sisto:metricrelhyp}.
\end{proof}

\subsection{Stable surface subgroups}

An undistorted subgroup is said to be \emph{(quasi-geodesically)
stable} if each pair of points in that subgroup are connected by uniformly Morse
quasigeodesics, see \cite{DurhamTaylor:stability}.

\begin{prop}\label{stablesurface}
    $W$ 
    contains a closed hyperbolic surface subgroup which is stable.
\end{prop}
\begin{proof} 
    Recall that in any right-angled Coxeter group an 
    induced subgraph yields a subgroup isomorpic to the right-angled Coxeter group 
    of the associated subgraph.  
    Also, note that the right-angled Coxeter group associated to a cycle of 
    length at least~5 is a 2--dimensional hyperbolic orbifold group.
    Thus, the subgraph spanned by $\PP$, which is a cycle of 
    length $m\geq 5$, yields a subgroup, $H$, which is isomorphic 
    to a 2--dimensional hyperbolic orbifold group. 
    
%     Recall that, associated to any $CAT(0)$ cube complex, Hagen
%     defined an associated hyperbolic metric space called the
%     \emph{contact graph}, see \cite{Hagen:quasi_arboreal}. Let $\fontact 
%     W$ denote the contact graph associated to the universal cover of 
%     the Davis complex of $W$. 
%     

%     The remainder of the argument can be proved explicitly, but, 
%     for brevity, we instead cite some recent general results which 
%     provide the results we need.
    
    By construction, the subgraph
    $\PP$ doesn't contain any pair of non-adjacent vertices in a
    common join of $\Gamma$ (by assumption we use the term 
    join to mean non-degenerate join,  
    in the sense that both parts of the join have
    diameter at least 2); thus no special subgroup of $H$ is a direct 
    factor in a 
    non-hyperbolic special subgroup of $W$.
 
    It is proven in \cite[Theorem~B]{AbbottBehrstockDurham} that, a
    subgroup of a hierarchically hyperbolic space is stable if and
    only if it is undistorted and has uniformly bounded projections to
    the curve graph of every proper domain in some hierarchically
    hyperbolic structure; in turn, this property is equivalent to the
    orbit of the subgroup being quasi-isometrically embedded in the
    curve graph of the nest-maximal domain.  
    
    A hierarchically
    hyperbolic structure on right-angled Coxeter groups was built in
    \cite{BehrstockHagenSisto:HHS1}; in that structure the curve
    graphs are contact graphs which were first defined in
    \cite{Hagen:quasi_arboreal} where it was proven they are all quasi-trees.  In
    the present context, a domain can be thought of as a certain convex
    subcomplex of the Davis complex arising from special subgroups and
    orthogonality holds when the Davis complex contains the direct
    product of a given such pair. 

    Since a closed surface group can not be quasi-isometrically embedded 
    in a quasi-tree, to verify the proposition 
    a different hierarchically hyperbolic structure 
    must be used; for this we now rely on a method developed in 
    \cite{AbbottBehrstockDurham} for modifying structures.
    The construction given in 
    \cite[Theorem~3.11]{AbbottBehrstockDurham} modifies the 
    structure, by removing certain domains and augmenting the curve 
    graphs associated to domains in which they are nested. Indeed, 
    passing to the new structure and back, it follows that to verify a 
    given subgroup has uniformly bounded projections in a particular 
    structure, it suffices to 
    verify 
    that the subgroup has uniformly bounded projections to the set of domains
    $\mathcal W$ with the property that for each $W\in \mathcal W$,
    there exists a domain $U$ which is orthogonal to $W$ and with
    infinite diameter curve graph (since all other domains will be 
    removed in the new structure while only changing the curve graph  
    of the nest-maximal domain).  
    
    In the standard hierarchically hyperbolic structure the only
    domains to which $H$ has unbounded projections are contained in
    the subcomplex of the Davis complex associated to the special
    subgroup $H$.  As noted above, since no pair of vertices in $\PP$
    are contained in a common join in $\Gamma$, these domains all have
    the property that any domains orthogonal to one of them must have
    uniformly bounded diameter.  Thus, applying the results of
    \cite{AbbottBehrstockDurham} just discussed, it follows that $H$ is
    stable.

    Now take a cover of the orbifold to get the desired 
    closed hyperbolic surface.
\end{proof}

\begin{rem}\label{rem:thanksref} We note that in the subgroup $H$ constructed in 
    Proposition~\ref{stablesurface}, it follows easily from 
    \cite[Proposition~3.3]{KapovichLeeb:3manifolds} 
    (or, more explicitly, from \cite[Theorem~1.1]{Tran:loxcox}), that 
    each infinite order element in $H$ acts 
    as a rank-one isometry of the Davis complex of $W$ and is thus 
    Morse. To see that such an argument alone is not enough to prove 
    stability of 
    the subgroup, we note that in  
    \cite{OlshanskiiOsinSapir:lacunary} they construct 
    lacunary hyperbolic groups which provide an obstruction. In 
    particular, in \cite[Theorem~1.12]{OlshanskiiOsinSapir:lacunary} 
    it is proven that there exist  
    infinite finitely generated non-hyperbolic groups in which every 
    proper subgroup is infinite cyclic and generate uniformly Morse 
    quasigeodesics. 
\end{rem}

The next result follows from Proposition~\ref{stablesurface} and 
\cite[Proposition~4.2]{Cordes:Morseboundary}.

\begin{cor}\label{embeddedcircle}
    The right-angled Coxeter group $W$ 
    contains a topologically embedded circle in its Morse boundary.
\end{cor}

\section{Applications}

\subsection{Morse boundaries}

Charney and Sultan introduced a boundary for CAT(0) groups which 
captures aspects of the negative curvature of the group  
\cite{CharneySultan}. Their construction was then generalized by 
Cordes to a framework which exists for 
all finitely generated groups \cite{Cordes:Morseboundary}; in this 
general context it is known as the \emph{Morse boundary}. 
Charney and Sultan built examples 
of relatively hyperbolic right-angled Coxeter groups whose 
boundaries are not totally disconnected \cite{CharneySultan}. More 
generally, it is now known that for hyperbolic groups, the 
Morse boundary coincides with the hyperbolic boundary 
\cite[Main Theorem (3)]{Cordes:Morseboundary}; using this it is easy to produce many examples of 
hyperbolic and relatively hyperbolic right-angled Coxeter groups 
whose boundary are not totally disconnected. 

On the other hand, the Morse boundary of any right-angled Artin group
$A_\Gamma$ is totally disconnected.  The two-dimensional case of this
is implicit in \cite{CharneySultan}.  In general, this fact follows from
the fact that the contact graph has
totally disconnected boundary (since it is a quasi-tree) and 
\cite[Theorem~F]{CordesHume:Stability}. To see this recall that 
\cite[Theorem~F]{CordesHume:Stability} provides a continuous
map from the Morse boundary of $A_\Gamma$ to the boundary of the
contact graph of the universal cover of the Salvetti complex of
$A_\Gamma$; at which point the result follows since 
Morse geodesic rays lying at infinite Hausdorff distance
cannot fellow-travel in the contact graph, so this continuous boundary
map is injective.

Accordingly, Ruth Charney and Alessandro Sisto 
raised the question of whether outside of the relatively 
hyperbolic setting, right-angled Coxeter groups all have totally 
disconnected Morse boundary. The group $W$ constructed 
above shows the answer is no, since it is not relatively hyperbolic 
by Proposition~\ref{quaddiv} and  
its boundary is not totally 
disconnected, by Corollary~\ref{embeddedcircle}.

\subsection{Stable subgroups}

Examples of stable subgroups are known both in the mapping class 
group \cite{Behrstock:asymptotic, DurhamTaylor:stability} and in right-angled Artin groups
\cite{KoberdaMangahasTaylor}. In both of these classes, 
all known examples of stable subgroups are 
virtually free; in the relatively hyperbolic setting on the other 
hand, it is easy to construct one-ended stable subgroups. 
Sam Taylor asked whether there exist non-relatively hyperbolic 
groups with one-ended stable subgroups. The example $W$ is a  
non-relatively hyperbolic 
group with one-ended stable subgroups by 
Proposition~\ref{quaddiv} and 
Proposition~\ref{stablesurface}.

\subsection{Commensurability}

A well-known construction of Davis--Januszkiewicz 
\cite{DavisJanuszkiewicz} shows that every right-angled Artin group  
is commensurable to some right-angled Coxeter group. 
The following is a well-known problem:

\begin{question}\label{question:raagracg}
    Which right-angled Coxeter groups are commensurable to 
    right-angled Artin groups?
\end{question}

It is known that 
any right-angled Artin groups either has divergence which is linear 
(if it is a direct product) or quadratic, see 
\cite{BehrstockCharney} or \cite{ABDDY:Pushing}.
Since divergence 
is invariant under quasi-isometry, and hence under commensurability as 
well, this puts a constraint on the 
answer to Question~\ref{question:raagracg}. Several people have 
raised the question of whether every right-angled Coxeter group with 
quadratic divergence is quasi-isometric to some right-angled Artin 
group. The group $W$ shows that the answer is no, since \cite[Main 
Theorem (2)]{Cordes:Morseboundary} proves that the Morse boundary 
is invariant under quasi-isometries, but Propositions~\ref{quaddiv} 
and~\ref{stablesurface} show that 
the group $W$ is a right-angled Coxeter group with quadratic 
divergence whose Morse boundary 
contains an embedded circle, while the Morse boundary of any 
right-angled Artin group is totally disconnected.

\section{Further questions}

In \cite{BehrstockHagenSusseFalgas:StructureRandom} it was established
that for a large range of density functions that asymptotically almost
surely the random graph yields a right-angled Coxeter group with
quadratic divergence.  More generally, it is known that among  
$n$--vertex graphs with density greater than $1/n$ and bounded away
from 1, that asymptotically almost surely any random graph of this 
type will contain a
large induced polygon.  Accordingly, we expect that the proof of
Corollary~\ref{embeddedcircle} could be used to verify:

\begin{conj}
    For any density greater than $1/n$ and bounded away from 1,  
    asymptotically almost surely the random right-angled Coxeter 
    group contains circles in its Morse boundary. In particular, it
    is not virtually a right-angled Artin group.
\end{conj}

The prevalence of right-angled Coxeter groups with quadratic 
divergence makes the following an appealing (but likely very 
difficult) question:
\begin{question}\label{question:racgclassification}
    Classify right-angled Coxeter groups with quadratic divergence up 
    to quasi-isometry. Classify them up to commensurability.
\end{question}

\medskip

The proof of Proposition~\ref{stablesurface} and 
Remark~\ref{rem:thanksref}  
indicate that there are some non-trivial constraints on 
what the set of domains could be in any potential 
hierarchically hyperbolic structure on a lacunary hyperbolic group. 
Accordingly we ask:

\begin{question}\label{question:lacunaryHHS}
    Is a lacunary hyperbolic group hierarchically hyperbolic if and 
    only if it is hyperbolic?   
\end{question}

Further, in light of Proposition~\ref{stablesurface} and Remark~\ref{rem:thanksref} 
it would be interesting if the following was true in general or 
under some moderate hypotheses:

\begin{question}\label{question:morseHHS}
    Let $G$ be a hierarchically hyperbolic group with a subgroup 
    $H$. If all infinite cyclic subgroups of $H$ are uniformly 
    Morse, does that imply that $H$ is stable?
\end{question}

%%%%%%%%%%%%%%%%%%%%%%%%%%%%%%%%
%\bibliographystyle{alpha} %
% \bibliographystyle{plain} %
% \bibliographystyle{/Users/jason/Documents/Math/zTEX.macros.symbols/ctmbib} %                       
% \bibliography{/Users/jason/Documents/Math/zTEX.macros.symbols/behrstock}

\begin{thebibliography}{BFRHS}

\bibitem[ABD]{AbbottBehrstockDurham}
Carolyn Abbott, Jason Behrstock, and Matthew~G Durham.
\newblock {Largest acylindrical actions and stability in hierarchically
  hyperbolic groups}.
\newblock \textsc{arXiv:1705.06219}.

\bibitem[ABD{\etalchar{+}}]{ABDDY:Pushing}
A.~Abrams, N.~Brady, P.~Dani, M.~Duchin, and R.~Young.
\newblock {Pushing fillings in right-angled {A}rtin groups}.
\newblock {\em J. Lond. Math. Soc. (2)} {\bf 87} (2013), 663--688.

\bibitem[Beh]{Behrstock:asymptotic}
J.~Behrstock.
\newblock {Asymptotic geometry of the mapping class group and {T}eichm\"{u}ller
  space}.
\newblock {\em Geometry \& Topology} {\bf 10} (2006), 1523--1578.

\bibitem[BC]{BehrstockCharney}
J.~Behrstock and R.~Charney.
\newblock {{Divergence and quasimorphisms of right-angled Artin groups}}.
\newblock {\em Math. Ann.} {\bf 352} (2012), 339--356.

\bibitem[BDM]{BehrstockDrutuMosher:thick}
J.~Behrstock, C.~Dru\c{t}u, and L.~Mosher.
\newblock {Thick metric spaces, relative hyperbolicity, and quasi-isometric
  rigidity}.
\newblock {\em Math. Ann.} {\bf 344} (2009), 543--595.

\bibitem[BFRHS]{BehrstockHagenSusseFalgas:StructureRandom}
J.~Behrstock, V.~Falgas-Ravry, M.~Hagen, and T.~Susse.
\newblock {Global structural properties of random graphs}.
\newblock {\em Int. Math. Res. Not.} (2016).
\newblock To appear.

\bibitem[BH]{BehrstockHagen:cubulated1}
J.~Behrstock and M.F.\ Hagen.
\newblock {Cubulated groups: thickness, relative hyperbolicity, and simplicial
  boundaries}.
\newblock {\em Geometry, Groups, and Dynamics} {\bf 10} (2016), 649--707.

\bibitem[BHS]{BehrstockHagenSisto:HHS1}
J.~Behrstock, M.F. Hagen, and A.~Sisto.
\newblock {{Hierarchically hyperbolic spaces I: curve complexes for cubical
  groups}}.
\newblock {\em Geometry \& Topology}, 21:1731--1804, 2017.

\bibitem[Cha]{Charney:RAAGsurvey}
Ruth Charney.
\newblock {An introduction to right-angled {A}rtin groups}.
\newblock {\em Geom. Dedicata} {\bf 125} (2007), 141--158.

\bibitem[CS]{CharneySultan}
Ruth Charney and Harold Sultan.
\newblock {Contracting boundaries of {$\rm CAT(0)$} spaces}.
\newblock {\em J. Topol.} {\bf 8} (2015), 93--117.

\bibitem[Cor1]{Cordes:Morseboundary}
Matthew Cordes.
\newblock {Morse boundaries of proper geodesic metric spaces}.
\newblock \textsc{arXiv:1502.04376}.%, 2015.

\bibitem[Cor2]{Cordes:Morsesurvey}
Matthew Cordes.
\newblock {A survey on Morse boundaries \& stability}.
\newblock \textsc{arXiv:1704.07598}.%, 2017.

\bibitem[CH]{CordesHume:Stability}
Matthew Cordes and David Hume.
\newblock {Stability and the Morse boundary}.
\newblock {\em To appear in Groups, Geometry, and Dynamics} (2016).

\bibitem[DaT]{DaniThomas:divcox}
Pallavi Dani and Anne Thomas.
\newblock {Divergence in right-angled {C}oxeter groups}.
\newblock {\em Trans. Amer. Math. Soc.} {\bf 367} (2015), 3549--3577.

\bibitem[DJ]{DavisJanuszkiewicz}
Michael~W. Davis and Tadeusz Januszkiewicz.
\newblock {Right-angled {A}rtin groups are commensurable with right-angled
  {C}oxeter groups}.
\newblock {\em J. Pure Appl. Algebra} {\bf 153} (2000), 229--235.

\bibitem[DT]{DurhamTaylor:stability}
Matthew~Gentry Durham and Samuel~J. Taylor.
\newblock {Convex cocompactness and stability in mapping class groups}.
\newblock {\em Algebr. Geom. Topol.} {\bf 15} (2015), 2839--2859.

\bibitem[Hag1]{Hagen:quasi_arboreal}
Mark~F. Hagen.
\newblock {Weak hyperbolicity of cube complexes and quasi-arboreal groups}.
\newblock {\em J. Topol.} {\bf 7} (2014), 385--418.

\bibitem[KL]{KapovichLeeb:3manifolds}
M.~Kapovich and B.~Leeb.
\newblock {{$3$}-manifold groups and nonpositive curvature}.
\newblock {\em Geom. Funct. Anal.} {\bf 8} (1998), 841--852.

\bibitem[KMT]{KoberdaMangahasTaylor}
Thomas Koberda, Johanna Mangahas, and Samuel~J Taylor.
\newblock {The geometry of purely loxodromic subgroups of right-angled Artin
  groups}.
\newblock \textsc{arXiv:1412.3663}.%, 2014.

\bibitem[Lev]{Levcovitz:CFSdiv}
Ivan Levcovitz.
\newblock {{Divergence of CAT(0) Cube Complexes and Coxeter Groups}}.
\newblock {arXiv:1611.04378}.%, 2016.

\bibitem[OOS]{OlshanskiiOsinSapir:lacunary}
Alexander~Yu. Ol$'$shanskii, Denis~V. Osin, and Mark~V. Sapir.
\newblock {Lacunary hyperbolic groups}.
\newblock {\em Geom. Topol.} {\bf 13}(2009), 2051--2140.
\newblock With an appendix by Michael Kapovich and Bruce Kleiner.

\bibitem[Sis]{Sisto:metricrelhyp}
A.~Sisto.
\newblock {{On metric relative hyperbolicity}}.
\newblock \textsc{arXiv:1210.8081}.%, 2012.

\bibitem[Tat2]{Hagen:tat}
Mark~F. Hagen.
\newblock {Tattoo}.
\newblock In preparation.

\bibitem[Tra]{Tran:loxcox}
Hung~Cong Tran.
\newblock {{Purely loxodromic subgroups in right-angled Coxeter groups}}.
\newblock \textsc{arXiv:1703.09032}.

\end{thebibliography}
%%%%%%%%%%%%%%%%%%%%%%%%%%%%%%%%

\newcommand{\etalchar}[1]{$^{#1}$}

\end{document}